\documentclass[11pt]{amsart}
\usepackage{geometry}                
\usepackage{graphicx}
\usepackage{amssymb}
\usepackage{amsmath}
\usepackage{epstopdf}
\usepackage{amsthm}

\newtheorem{definition}{Definition} 
\newtheorem{prop}[definition]{Proposition} 
\newtheorem{thm}[definition]{Theorem} 
\newtheorem{lem}[definition]{Lemma} 
\newtheorem{cor}[definition]{Corollary}

\def\AAA{\mathcal{A}}

\def\BBB{\mathcal{B}}
\def\C{\mathcal{C}}
\def\D{\Delta}
\def\D'{\Delta}

\def\H{\mathcal{H}}
\def\HHH{\mathcal{H}}

\def\NN{\mathbb{N}}

\def\P{\mathcal{P}}

\def\R{\mathbb{R}}
\def\RR{\mathbb{R}}

\def\Z{\mathbb{Z}}
\def\ZZ{\mathbb{Z}}

\def\x{\mathbf{x}}

\def\rar{\rightarrow}

\def\pull{\mathrm{pull}}
\def\conv{\mathrm{conv}\:}

\def\head{\mathrm{head}}
\def\tail{\mathrm{tail}}

\newcommand{\Tension}{\theta}
\newcommand{\Flow}{\varphi}
\newcommand{\mFlow}{\bar{\Flow}}
\newcommand{\mTension}{\bar{\Tension}}

\newcommand{\sprod}[2]{\langle {#1} , {#2} \rangle}
\newcommand{\defn}[1]{\textbf{#1}}
\newcommand{\floor}[1]{\lfloor {#1} \rfloor}
\newcommand{\ceil}[1]{\lceil {#1} \rceil}

\newcommand{\inter}{\mathrm{int }}

\begin{document}

\title{Bounds on the Coefficients of Tension and Flow Polynomials}

\author{Felix Breuer}
\address{Felix Breuer, Freie Universit\"at Berlin, Arnimallee 3, 14195 Berlin, Germany. }
\thanks{Felix Breuer was supported by Emmy Noether grant HA 4383/1 of the German Research Foundation (DFG).}
\email{felix.breuer@fu-berlin.de}

\author{Aaron Dall}
\address{Aaron Dall, Wartburgstr. 10, 10823 Berlin, Germany. }
\email{adall1979@gmail.com}

\begin{abstract}
The goal of this article is to obtain bounds on the coefficients of modular and integral flow and tension polynomials of graphs. To this end we make use of the fact that these polynomials can be realized as Ehrhart polynomials of inside-out polytopes. Inside-out polytopes come with an associated relative polytopal complex and, for a wide class of inside-out polytopes, we show that this complex has a convex ear decomposition. This leads to the desired bounds on the coefficients of these polynomials.
\end{abstract}

\maketitle

\section{Introduction}

The goal of this article is to obtain bounds on the coefficients of modular and integral flow and tension polynomials of graphs. To this end we employ Ehrhart theory, the theory of lattice points in polyhedra. In past research \cite{BeckZaslavsky06a, BeckZaslavsky06b, Dall08, BreuerSanyal09, Breuer09} it has been shown that each of these polynomials can be realized as the Ehrhart polynomial of inside-out polytopes. Inside-out polytopes come with an associated relative polytopal complex $\C'\subset\C$. In this article we show that for a wide class of inside-out polytopes, this polytopal complex can be triangulated such that resulting relative simplicial complex $\Delta'\subset\Delta$ is unimodular and $\Delta'$ admits a convex ear decomposition (Theorem~\ref{thm:iop-ced}). This implies constraints on the Ehrhart polynomial of the inside-out polytope (Theorem~\ref{thm:iop-bounds}). In the case of modular flow and tension polynomials, which is of greater interest in graph theory than the integral case, this leads to upper and lower bounds on the coefficients of these polynomials (Theorem~\ref{thm:bounds-modular}). In the integral case, Theorem~\ref{thm:iop-bounds} calls for a closer analysis of the Ehrhart polynomials of flow and tension polytopes, which we begin in Section~\ref{sec:integral-case}. In particular we obtain upper and lower bounds on the $h^*$-vector of tension and flow polytopes.

This article is related to similar work done for the chromatic polynomial of a graph. Steingr\'{\i}msson \cite{Steingrimsson01} showed that the chromatic polynomial of graph can be realized as the Hilbert function of a certain square-free monomial ideal, which in turn gave rise to what  Steingr\'{\i}msson called the coloring complex of a graph. Subsequent articles, building on Steingr\'{\i}msson's work, have investigated various properties of the coloring complex.  Jonsson \cite{Jonsson05} showed the coloring complex to be constructible and hence Cohen-Macaulay. This result was improved by Hultman \cite{Hultman07} who showed the coloring complex to be shellable and by Hersh and Swartz \cite{HershSwartz08} who showed that the coloring complex has a convex ear decomposition. These results translate into bounds on the coefficients of the chromatic polynomial. 

Breuer and Dall \cite{BreuerDall09} have shown that the integral and modular flow and tension polynomials of a graph can also be realized as Hilbert functions of square-free monomial ideals, working with the theory of lattice polytopes rather than the theory of Stanley-Reisner rings. In the present article we continue this work to obtain bounds on the coefficients of these polynomials. As we only need to focus on the geometry to obtain our results, we do not use that these are Hilbert functions but only that they are Ehrhart polynomials.

This article is organized as follows. We begin in the realm of discrete geometry where we use Section~\ref{sec:disc-geom} to give some preliminary definitions. We then focus on regular subdivisions of polytopes in Section~\ref{sec:regular-subdivisions} where we show that the subdivision determined by an inside-out polytope is regular (Lemma~\ref{lem:iop-regular}). We introduce the four counting polynomials in Section~\ref{sec:flow-and-tension} and summarize the realization of these as Ehrhart polynomials of inside-out polytopes. In Section~\ref{sec:ced} we define convex ear decompositions and show that regular triangulations of the complexes associated with inside-out polytopes have a convex ear decomposition. As we need to relate the $h$-vector of an abstract simplicial complex to the Ehrhart $h^*$-vector of a relative polytopal complex, we give an overview of the well-known relationships between $f$-, $h$- and $h^*$-vectors in Section~\ref{sec:fhh-vectors}. We are then in the position to derive our main results in Section~\ref{sec:enumerative-consequences}. To deal with the integral case, more work has to be done, which we begin in Section~\ref{sec:integral-case} by obtaining bounds on the $h^*$-vectors of flow and tension polytopes.

\section{Preliminaries from Discrete Geometry}
\label{sec:disc-geom}

Before we begin, we gather some definitions from discrete geometry. We recommend the textbooks \cite{BeckRobins07,Ziegler95} as references.

The \defn{Ehrhart function} $L_A$ of any set $A\subset \RR^n$ is defined by $L_A(k)=|\ZZ^n\cap k\cdot A|$ for $k\in\NN$. A \defn{lattice polytope} is a polytope in $\RR^n$ such that all vertices are integer points. It is a theorem of Ehrhart that the Ehrhart function $L_P(k)$ of a lattice polytope is a polynomial in $k$. Two polytopes $P,Q$ are \defn{lattice isomorphic}, $P\approx Q$, if there exists an affine isomorphism $A$ such that $A|_{\ZZ^n}$ is a bijection onto $\ZZ^n$ and $AP=Q$. A \defn{$d$-simplex} is the convex hull of  $d+1$ affinely independent points. A $d$-simplex is \defn{unimodular} if it is lattice isomorphic to the convex hull of $d+1$ standard unit vectors. A \defn{hyperplane arrangement} is a finite collection $\HHH$ of affine hyperplanes and $\bigcup\HHH$ denotes the union of all of these. A \defn{cell} of $\HHH$ is the closure of a component of the complement of $\bigcup\HHH$. We make use of the notation 
$$ H_{a,b} = \{ x\in\RR^n \;|\; \sprod{a}{x} = b\}$$
for $a\in\RR^n$ and $b\in \RR$ to denote hyperplanes.

A \defn{polytopal complex} is a finite collection $\C$ of polytopes in some $\RR^n$ with the following two properties: If $P\in\C$ and $F$ is a face of $P$, then $F\in\C$; and if $P,Q\in\C$ then $F=P\cap Q\in \C$ and $F$ is common face of both $P$ and $Q$. The polytopes in $\C$ are called \defn{faces} and $\bigcup\C$, the union of all faces of $\C$,  is called the \defn{support} of $\C$. A (geometric) \defn{simplicial complex} is a polytopal complex in which all faces are simplices. An \defn{abstract simplicial complex} is a set $\Delta$ of subsets of a finite set $V$ such that $\Delta$ is closed under taking subsets. A geometric simplicial complex gives rise to an abstract simplicial complex and every abstract simplical complex can be realized by a geometric one. We will not distinguish notationally between these two notions as it will be clear from the context which is meant. A polytopal complex $\C'$ that is a subset $\C'\subset\C$ of a polytopal complex $\C$ is called a \defn{subcomplex} of $\C$. Subcomplexes of abstract simplical complexes are defined similarly. Given a collection $S$ of polytopes in $\RR^n$ such that for any $P,Q\in S$ the set $P\cap Q$ is a face of both $P$ and $Q$, the polytopal complex $\C$ \defn{generated} by $S$, is $\C=\{F | \text{$F$ a face of $P\in S$}\}$. A \defn{subdivision} of a polytopal complex $\C$ is a polytopal complex $\C'$ such that $\bigcup\C=\bigcup\C'$ and every face of $\C'$ is contained in a face of $\C$. A \defn{triangulation} is a subdivision in which all faces are simplicies. A \defn{unimodular triangulation} is a triangulation in which all simplices are unimodular.

A \defn{relative polytopal complex} $\C'\subset\C$ is pair of polytopal complexes $\C'$ and $\C$ such that $\C'$ is a subcomplex of $\C$. An \defn{inside-out polytope} is a pair $(P,\HHH)$ of a polytope $P\subset\RR^n$ and a finite collection of hyperplanes $\HHH$ in $\RR^n$ such that each hyperplane $H\in\HHH$ meets the relative interior of $P$. Every inside-out polytope $(P,\HHH)$ comes with an associated relative polytopal complex $\C'\subset \C$, where $\C$ is generated by the cells of the hyperplane arrangement intersected with $P$ and $\C'$ is the subcomplex of $\C$ consisting of all those faces of $\C$ that are contained in $\bigcup \HHH\cup \partial P$. The Ehrhart function $L_{(P,\HHH)}$ of  an inside-out poltyope is the Ehrhart function of $\bigcup\C\setminus \bigcup\C'$, i.e., 
$$L_{(P,\HHH)}(k)=|\ZZ^n\cap k\cdot(\inter\: P \setminus \bigcup\HHH)|$$ 
for $k\in\NN$. We call an inside-out polytope \defn{integral} if all the faces of $\C$ are lattice polytopes. In this case $L_{(P,\HHH)}(k)$ is a polynomial in $k$.

\section{Regular Subdivisions}
\label{sec:regular-subdivisions}

We are going to use the concept of a regular subdivision of a polytope. For a thorough treatment of this concept we refer to \cite{BrunsGubeladze09, Sturmfels96, Lee2004}. In this section we show that the polytopal complexes induced by inside-out polytopes are regular subdivisions. Moreover we present two tools from the literature that we are going to use.  This will enable us to show in Section~\ref{sec:ced} that the complexes that interest us have convex ear decompositions.

Loosely speaking, a polytopal complex $\P$ is a regular subdivision of a $d$-polytope $P$ if it is a subdivision of $P$ and there exists a $(d+1)$-polytope $Q$ such that $\P$ is the projection of the lower hull of $Q$. More precisely, we let $P$ be a $d$-polytope and assume without loss of generality that $P\subset\R^d\times\{0\}\subset\R^{d+1}$ and define $\pi$ to be the map $(v_1,\ldots,v_d,v_{d+1})\mapsto (v_1,\ldots,v_d,0)$ that projects away the last coordinate. A polytopal complex $\P$ with $\bigcup \P=P$ is a \defn{regular subdivision} of $P\subset\R^n$ if there exists a $(d+1)$-polytope $Q\subset\R^{d+1}$ such that the facets of $P$ are precisely the images under $\pi$ of those facets of $Q$ whose outer normal $u=(u_1,\ldots,u_{d+1})$ has $u_{d+1}<0$. A \defn{regular triangulation} is a regular subdivision that is a triangulation.

We note that having a regular subdivision is hereditary in the sense that if $\C$ is a subcomplex of $\P$ such that $\bigcup\C$ is a polytope and $\P$ is a regular subdivision of $\bigcup\P$, then $\C$ is a regular subdivision of $\bigcup\C$.

If we can show that a given simplicial complex is a regular triangulation, this tells us that we can realize this complex as a subcomplex of the boundary of a simplicial polytope.

\begin{lem}
\label{lem:realizing-regular-triangulations}
Let $P$ be a $d$-polytope and $\P$ a regular triangulation of $P$. Then:
\begin{enumerate}
\item \label{itm:bdrycomplex} $\P|_{\partial P}$ is combinatorially equivalent to the boundary complex of a simplicial $d$-polytope.
\item \label{itm:subcomplex} $\P$ is combinatorially equivalent to a subcomplex of the boundary complex of a simplicial $(d+1)$-polytope.
\end{enumerate}
\end{lem}

\begin{proof}
(\ref{itm:bdrycomplex}) is a result by Bruns and R\"omer \cite[Lemma 9]{BrunsRomer07}. To show (\ref{itm:subcomplex}) we claim that there exists a $(d+1)$-polytope $P'$ and a regular triangulation $\P'$ of $P'$ such that $\P$ is combinatorially equivalent to a subcomplex of $\P'|_{\partial P'}$. Then we can apply (\ref{itm:bdrycomplex}) to complete the proof of (\ref{itm:subcomplex}).

To show the above claim, we argue as follows. By definition of a regular subdivision, $\P$ is the complex defined by the lower hull of some polytope $P'$. Let $\P'$ be any regular triangulation of $P'$ such that $\P'|_{\partial P'}$ is a refinement without new vertices of the boundary complex of $P'$. $\P$ is combinatorially equivalent to a subcomplex of the boundary complex of $P'$ and as $\P$ is simplicial, $\P$ is also combinatorially equivalent to a subcomplex of $\P'|_{\partial P'}$.
\end{proof}

To give a formal proof of the fact that a given complex is indeed a regular subdivision, it is convenient to work with a different definition of regular subdivision, that is easily seen to be equivalent: Let $\C$ be a polytopal complex whose support is a polytope $P$. A function $\omega:P\rar\R$ is a \defn{$\C$-linear strictly $\C$-convex support function} if the following properties hold.
\begin{enumerate}
\item $\omega$ is continuous.
\item $\omega$ is affine on each $F\in\C$.
\item $\omega$ is convex on $P$.
\item \label{itm:dom-of-lin} The convex sets $S\subset P$ that are inclusion maximal with the property that there exists an affine function $f$ such that $f\leq \omega$ and $f|_S=\omega|_S$, are exactly the facets of $\C$.
\end{enumerate}
We call a convex set $S$ with the properties given in (\ref{itm:dom-of-lin})  a \defn{domain of linearity}. A subdivision $\C$ of $P$ possesses an $\C$-linear strictly $\C$-convex support function if and only if $\C$ is a regular subdivision of $P$.

\begin{lem}
\label{lem:iop-regular}
Let $(P,\H)$ be an inside-out polytope and let $\C'\subset\C$ be the associated relative polytopal complex. Then $\C$ is a regular subdivision of $P$.
\end{lem}

\begin{proof}
For every affine hyperplane $H\in\H$ we define a function $\omega_H:P\rar \R$ as follows. Let $z$ be any point in $H$, let $v$ be a normal vector of $H$ and define
\[
  \omega_H(x) := |\sprod{v}{z} - \sprod{v}{x}|.
\]
As can be checked easily, $\omega_H$ is continuous and affine on each of the two half spaces defined by $H$. Moreover $\omega_H$ is convex, i.e.,\
\begin{eqnarray}
\label{eqn:convexity}
  \omega_H(\lambda x+ (1-\lambda) y)) & \leq & \lambda\omega_H(x)+(1-\lambda)\omega_H(y),
\end{eqnarray}
where the inequality above is strict if and only if $x,y\not\in H$ lie in opposite half-spaces.

Now, we define a function $\omega:P\rar\R$ by
\[
  \omega(x) = \sum_{H\in\H} \omega_H(x).
\]
We claim that this function is a $\C$-linear strictly $\C$-convex support function.
\begin{enumerate}
\item $\omega$ is continuous because it is a sum of continuous functions.
\item $\omega$ is affine on each $F\in\C$, because each of the $\omega_H$ is affine on $F$.
\item $\omega$ is convex because it is a sum of convex functions.
\end{enumerate}
All that we have left to show is property (\ref{itm:dom-of-lin}). 

Let $S\subset P$ be a convex set that is inclusion maximal with respect to the property that there is an affine function $f:\bigcup\C\rar\R$ with $f\leq \omega$ and $f|_S=\omega|_S$. Assume that there exist $x,y\in S$  that are contained in the relative interior of two different maximal faces of $\C$. That means that there is a hyperplane $H_0\in\H$ that separates the two. Then we have
\[
    f(\frac{1}{2}x+\frac{1}{2}y) = \omega(\frac{1}{2}x+\frac{1}{2}y) < \frac{1}{2} \omega(x) + \frac{1}{2}\omega(y) = \frac{1}{2}f(x) + \frac{1}{2}f(y)
\]
because inequality (\ref{eqn:convexity}) holds strictly for $\omega_{H_0}$ and weakly for all other $\omega_H$. But this means that $f$ is not affine on $S$, a contradiction. We conclude that no two points in $S$ can be separated by a hyperplane in $\H$, which shows that $S$ must be contained in a maximal face $F\in\C$. 

To show that $F$ is contained in $S$ consider the following. We already know that $\omega|_F$ is an affine function. Because $F$ is full-dimensional, $\omega|_F$ can be extended uniquely to an affine function $f:P\rar\R$ and we have $f|_F=\omega|_F$ by construction. Let $x\in P$ be any point. Pick a point $y\not=x$ that lies in the relative interior of $F$ and choose $\lambda\in(0,1)$ such that $\lambda x+(1-\lambda)y$ also lies in the relative interior of $F$. This is possible because $P$ is convex and $F$ is full-dimensional. Then
\begin{eqnarray*}
\lambda f(x) + (1-\lambda)f(y) & = & f(\lambda x + (1-\lambda) y) \\
& = & \omega(\lambda x + (1-\lambda) y) \\
& \leq & \lambda \omega(x) + (1-\lambda) \omega (y).
\end{eqnarray*}
Because $f(y)=\omega(y)$ this implies $f(x)\leq\omega(x)$. Thus $f$ demonstrates that $F$ is a domain of linearity. As $S$ was defined to be inclusion maximal, we can conclude $F=S$.

We have now shown that every such set $S$ is a maximal face of $\C$. But the above argument also shows that any maximal face $F\in\C$ is a domain of linearity. Therefore the proof is complete.
\end{proof}

Of course $\C$ is not going to be simplicial in general. To obtain a simplicial complex, we need to refine $\C$ further. To this end we use the concept of a pulling refinement, see \cite{Sturmfels96}: Given a polytopal complex $\C$ and a vertex $v$ of $\C$ we define the complex $\pull(\C,v)$, obtained by \defn{pulling $v$}, by
\[
  \pull(\C,v) = \{ F\in\C | v\not\in F\} \cup \bigcup_Q \{ \conv(F\cup v) | \text{$F$ a face of $Q$}, v\not\in F \}
\]
where the union runs over all $Q\in\C$ such that $v\in Q$. A \defn{pulling refinement} of $\C$ is a polytopal complex obtained by pulling several vertices of $\C$ in a given order. A \defn{pulling triangulation} of $\C$ is a pulling refinement of $\C$ that is a triangulation. One important property of this method of refinement is that pulling a vertex preserves regularity of the subdivision: If $\C$ is a regular subdivision of a polytope $P$ and $\C'$ is a pulling refinement of $\C$, then $\C'$ is a regular subdivision of $P$. See \cite{HaasePaffenholz06, Lee2004}.

A lattice polytope $P$ such that all pulling triangulations of $P$ are unimodular is called \defn{compressed}. An integral inside-out polytope $P$ with associated relative polytopal complex $\C'\subset\C$ is compressed, if all faces of $\C$ are compressed.

\section{Flow and Tension Polynomials as Ehrhart Functions}
\label{sec:flow-and-tension}

Both the modular and integral variants of the flow and tension polynomials of a graph can be represented as Ehrhart functions of inside-out polytopes. In this section we define these polynomials and summarize the constructions of the corresponding inside-out polytopes. For the general graph-theoretic background we refer the reader to \cite{West01}. A detailed treatment of the material in this section can be found in \cite[Chapters 3-4]{Breuer09}.

Let $G$ be a directed graph with edge set $E$ and vertex set $V$.\footnote{In general $G$ may have multiple edges and/or loops. However, when studying flows, we exclude graphs that have bridges and when studying tensions, we exclude graphs that have loops, as there are no nowhere-zero flows and tensions in these cases, respectively.} A \defn{spanning forest} $T$ of $G$ is a maximal cycle-free spanning subgraph of $G$. With any path $P$ in the underlying undirected graph we can associate a sign vector $\sigma\in\{0,\pm1\}^E$ by letting $\sigma_e=+1$ if $e\in P$ and the orientation of $e$ and the direction of traversal of $e$ are the same, by letting $\sigma_e=-1$ if $e\in P$ and the orientation of $e$ and the direction of traversal of $e$ are opposite, and by letting $\sigma_e=0$ if $e\not\in P$.

Let $A$ be the vertex-edge incidence matrix of $G$. We will interpret $A$ as both a matrix with entries in $\ZZ$ and as a matrix with entries in $\ZZ_k$. In the former case, $\ker A\subset  \RR^E$ is the \defn{flow space} of $G$ and vectors $f\in \ker A\cap \{-k+1,\ldots,k-1\}^E$ are called \defn{$k$-flows} of $G$. In the latter case, vectors $f\in \ker A\subset \ZZ_k^E$ are called \defn{$\ZZ_k$-flows}. We will identify integers with their respective cosets in $\ZZ_k$ and cosets in $\ZZ_k$ with their canonical representatives in $\ZZ$. Thus we may view the set of $k$-flows as being contained in the open cube $(-k,k)^E$ and the set of $\ZZ_k$-flows as being contained in the half-open cube $[0,k)^E$. 

Let $M$ denote the cycle-edge incidence matrix, i.e., the matrix whose row vectors are the sign vectors of all cycles of $G$. Again we interpret $M$ as both a matrix with entries in $\ZZ$ and as a matrix with entries in $\ZZ_k$. In the former case, $\ker M\subset  \RR^E$ is the \defn{tension space} of $G$ and vectors $t\in \ker M\cap \{-k+1,\ldots,k-1\}^E$ are called \defn{$k$-tensions} of $G$. In the latter case, vectors $t\in \ker A\subset \ZZ_k^E$ are called \defn{$\ZZ_k$-tensions}. Again we may view the set of $k$-tensions as being contained in the open cube $(-k,k)^E$ and the set of $\ZZ_k$-tensions as being contained in the half-open cube $[0,k)^E$. 

We call a vector $z$, with entries in $\ZZ$ or in $\ZZ_k$, \defn{nowhere-zero} if $z_e\not=0$ for all $e\in E$. Now we define for all $k\in\NN$
\begin{eqnarray*}
\Flow_G(k) & = & \# \{\text{nowhere-zero $k$-flows of $G$}\} \\
\mFlow_G(k) & = & \# \{\text{nowhere-zero $\ZZ_k$-flows of $G$}\}\\
\Tension_G(k) & = & \# \{\text{nowhere-zero $k$-tensions of $G$}\} \\
\mTension_G(k) & = & \# \{\text{nowhere-zero $\ZZ_k$-tensions of $G$}\}.
\end{eqnarray*}
It turns out that all of these functions are polynomials in $k$. The polynomials $\mFlow_G$ and $\mTension_G$ are called the \defn{modular flow and tension polynomials} of $G$, respectively. These are the classic graph polynomials defined by Tutte. They are evaluations of the Tutte polynomial and can be computed recursively, using a deletion-contraction formula. The modular tension polynomial, $\mTension_G$, is a non-trivial divisior of the chromatic polynomial of $G$. $\Flow_G$ and $\Tension_G$ are called the \defn{integral flow and tension polynomials} of $G$, respectively.  That these are in fact polynomials is a relatively recent result by  Kochol \cite{Kochol02}. They were studied more intensively in \cite{Chen07} and \cite{Dall08}.

From the above definitions it is straightforward to see that the integral flow and tension polynomials of a graph are Ehrhart functions of inside-out polytopes. See \cite{BreuerDall09, Breuer09} for details. Let $\HHH$ denote the hyperplane arrangement consisting of all coordinate hyperplanes $H_{e_i,0}$. Then
\[
    \Flow_G(k) = L_{\ker A\cap(-1,1)^E, \HHH}(k) \text{ and }  \Tension_G(k) = L_{\ker M\cap(-1,1)^E, \HHH}(k).
\]
These constructions are from \cite{BeckZaslavsky06a,BeckZaslavsky06b,Dall08}. The inside-out polytopes $(\ker A\cap(-1,1)^E, \HHH)$ and $(\ker M\cap(-1,1)^E, \HHH)$ are integral (which follows from the total unimodularity of $A$ and $M$, see \cite{Schrijver86}) and compressed (which follows e.g. from  a theorem by Ohsugi and Hibi \cite[Theorem 1.1]{OhsugiHibi01}).

To obtain the modular flow and tension polynomials as Ehrhart functions of inside-out polytopes, we have a bit more work to do. Let $T$ be a spanning forest of $G$. To every non-tree edge $e\in E\setminus T$ with $\tail(e)=u$ and $\head(e)=v$ there corresponds a unique path $P$ in $T$ from $v$ to $u$. Let $\sigma^e$ denote the sign vector of this path. Let $C$ denote the $|T|\times|E\setminus T|$-matrix that has the vectors $\sigma^e|_{T}$ as columns.  Any flow $f$ on $G$ is uniquely determined by $f|_{E\setminus T}$ via $f|_{T}= Cf|_{E\setminus T}$. Any tension $t$ on $G$ is uniquely determined by $t|_{T}$ via $t|_{E\setminus T} = (-C^t)t|_{T}$. We can therefore parameterize nowhere-zero $\ZZ_k$-flows ($\ZZ_k$-tensions) by lattice points in the open unit cube $(0,k)^{E\setminus T}$ (resp. $(0,k)^{T}$), subject to certain constraints that can be expressed via the matrix $C$. Let $\AAA$ denote the set of rows of $C$ and let $\BBB$ denote the set of rows of $-C^t$. Let $\HHH_\AAA$ denote the set of hyperplanes $H_{a,k}$  where $a\in\AAA$ and $k\in\ZZ$ such that $H_{a,k}$ meets $\inter \ (0,1)^{E\setminus T}$. Let $\HHH_\BBB$ denote the set of hyperplanes $H_{b,k}$  where $b\in\BBB$ and $k\in\ZZ$ such that $H_{b,k}$ meets $\inter \ (0,1)^{T}$. Then
\[
  \mFlow_G(k) = L_{(0,1)^{E\setminus T}, \HHH_\AAA}(k) \text{ and }  \mTension_G(k) = L_{(0,1)^{T}, \HHH_\BBB}(k).
\]
The inside-out polytopes $((0,1)^{E\setminus T}, \HHH_\AAA)$ and $((0,1)^{T}, \HHH_\BBB)$ are integral (which follows from the total unimodularity of $C$, see \cite{Schrijver86}) and compressed (which follows from e.g.\ Paco's Lemma, see \cite[Proposition 1.8]{HaasePaffenholz06}). The above constructions were given by Breuer and Sanyal in \cite{BreuerSanyal09,Breuer09}, to which we refer the interested reader for details. See also \cite{BreuerDall09}.

For any spanning forest $T$, the matrix $C$ defined above can also be used to construct inside-out polytopes in the integral case. For a graph $G$ with a fixed spanning forest $T$ we define the \defn{tension polytope} $T_G$ by
\[
T_G=\{t\in \R^T \mid -1\leq C^t t \leq 1, -1\leq t \leq 1\}
\]
and the \defn{flow polytope} $F_G$ by
\[
F_G=\{f\in \R^{E\setminus T} \mid -1\leq -C f \leq 1, -1\leq f \leq 1\}.
\]
Let $\HHH_T$ denote the collection of hyperplanes in $\R^T$ given by $C^t t=0$ together with the coordinate hyperplanes. Then the inside-out polytope $(T_G,\HHH_T)$ is a lattice transform of the inside-out polytope given above for the integral tension case and in particular $L_{(T_G,\HHH_T)}=\Tension_G$. Let $\HHH_F$ denote the collection of hyperplanes in $\R^{E\setminus T}$ given by $C f=0$ together with the coordinate hyperplanes. Then the inside-out polytope $(F_G,\HHH_F)$ is a lattice transform of the inside-out polytope given above for the integral flow case and in particular $L_{(F_G,\HHH_F)}=\Flow_G$. In \cite{FelsnerKnauer08}, the authors prove that (certain unimodular transformations of) generalized tension polytopes give rise to all distributive polytopes and give a connection to alcoved polytopes.

We summarize the content of this section in the following theorem.

\begin{thm}
For every graph $G$, the modular and integral flow and tension polynomials of $G$ can be realized as Ehrhart polynomials of compressed, integral inside-out polytopes.
\end{thm}

\section{Convex Ear Decompositions}
\label{sec:ced}

\begin{definition}
Let $\Delta$ be a $(d-1)$-dimensional simplicial complex. 
A convex ear decomposition of $\Delta$ is an ordered sequence $\Delta_1, \Delta_2,\dots, \Delta_m$ of pure $(d-1)$-dimensional subcomplexes of $\Delta$ such that
\begin{enumerate}
\item $\Delta_1$ is the boundary complex of a $d$-polytope. For each $j\ge 2$, $\Delta_j$ is a $(d-1)$-ball which is a proper subcomplex of the boundary complex of a simpicial $d$-polytope.
\item\label{boundary condition} $\partial \Delta_j =  \Delta_j \cap \bigcup_{i < j} \Delta_i$ for $j\ge2$.
\item\label{union condition} $\Delta = \bigcup_j \Delta_j$.
\end{enumerate}
\end{definition}

Convex ear decompositions are of interest because the existence of a convex ear decomposition of a complex $\Delta$ implies bounds on the $h$-vector of $\Delta$, see Theorem~\ref{thm:ced-bounds} below. To be able to apply this result, we now establish that inside-out polytopes have convex ear decompositions in the following sense.

\begin{thm}
\label{thm:iop-ced}
Let $(P,\H)$ be an inside-out polytope and let $\Delta'\subset\Delta$ be a regular triangulation of the associated relative polytopal complex. Then both $\Delta'$ and $\Delta'|_{\partial P}$ have a convex ear decomposition.
\end{thm}

\begin{proof}
We show only that $\Delta'$ has a convex ear decomposition. The proof to show that $\Delta'|_{\partial P}$ has a convex ear decomposition is analogous.

Let $\C'\subset \C$ denote the relative polytopal complex associated with the inside-out polytope $(P,\H)$. By Lemma~\ref{lem:iop-regular}, $\C$ is a regular subdivision of $P$. By assumption $\Delta$ is a regular triangulation of $\C$. By transitivity $\Delta$ is a regular triangulation of $P$. By Lemma~\ref{lem:realizing-regular-triangulations}.(\ref{itm:bdrycomplex}) the complex $\Delta|_{\partial P}$ is combinatorially equivalent to the boundary complex of a simplicial $d$-polytope, where $d=\dim P$. So we can define the first element in our convex ear decomposition to be $\Delta_0:=\Delta|_{\partial P}$.

Next we fix a total order $H_1,\ldots, H_l$ of the hyperplanes in $\H$. For any hyperplane $H$, we denote the positive and negative closed half-spaces by $H^+$ and $H^-$, respectively. Now for any $1\leq j \leq l$ and any function $\sigma:\{1,\ldots,j-1\}\rar\{+,-\}$ we define the set $P_j^\sigma$ to be the intersection
\[
  P_j^\sigma := P \cap \bigcap_{1\leq k < j} H^{\sigma(k)}_k.
\]
Note that $P_j^\sigma$ is either empty or $d$-dimensional as the hyperplanes all meet the interior of $P$. Next we define
\[
  C_{j}^{\sigma} := P_j^\sigma \cap H_j \text{ and } \Delta_j^\sigma := \Delta|_{C_j^\sigma}.
\]
For a given $j$, we denote the set of all $\sigma$ such that $\Delta_j^\sigma$ is $(d-1)$-dimensional by $S_j$ and we equip $S_j$ with an arbitrary total order $\prec$.

The complete set of ears for our convex ear decomposition is
\[
  \{\Delta_0\} \cup \{ \Delta_j^\sigma | 1\leq j \leq l, \sigma\in S_j \}.
\]
The total order $\prec$ we impose on this set is defined as follows. $\Delta_0$ is its minimal element. Moreover $\Delta_{j_1}^{\sigma_1} \prec \Delta_{j_2}^{\sigma_2}$, if $j_1 < j_2$ or if $j_1=j_2$ and $\sigma_1\prec\sigma_2$. Now we have to show that this really gives a convex ear decomposition.

(1) We note that for every $j$ and $\sigma$,  the complex $\Delta_j^\sigma$ is a regular triangulation of the $(d-1)$-dimensional polytope $C_j^\sigma$. ($\C'|_{C_j^\sigma}$ is a regular subdivision of $C_j^\sigma$ and $\Delta'|_{C_j^\sigma}$ is a regular triangulation of $\C'|_{C_j^\sigma}$.) Thus it is a pure $(d-1)$-dimensional simplicial complex which is a $(d-1)$-dimensional ball and by Lemma~\ref{lem:realizing-regular-triangulations}.(\ref{itm:subcomplex}) it is combinatorially equivalent to a proper subcomplex of the boundary of a simplicial $d$-polytope. Moreover it is a subcomplex of $\Delta$.

(2) Next we observe that for every $j$
\[
  \Delta_0\cup \bigcup_{1\leq k <j}\bigcup_{\sigma\in S_k} \Delta_j^\sigma = \partial P \cup \bigcup_{1\leq k < j} (H_k\cap P)
\]
which is a superset of $\partial \Delta_j^\sigma$ for every $\sigma\in S_j$.
On the other hand, no interior point of  $\Delta_j^\sigma$ is contained in $\partial P$ or any of the hyperplanes $H_k$ with $k<j$. Moreover, no interior point of $\Delta_j^\sigma$ is contained in any $\Delta_j^{\sigma'}$ for $\sigma'\not=\sigma$. Thus 
\[
  \partial  \Delta_j^\sigma =  \Delta_j^\sigma \cap (\Delta_0\cup \bigcup_{\substack{k<j\\ \sigma\in S_k}} \Delta_k^\sigma \cup \bigcup_{\sigma'\prec\sigma} \Delta_j^{\sigma'}).
\]

(3) Finally we note that on the level of sets
\[
 \Delta_0\cup \bigcup_{1\leq j \leq l}\bigcup_{\sigma\in S_j} \Delta_j^\sigma = \partial P \cup \bigcup_{1\leq j\leq l} H_j\cap P = \Delta'
\]
and thus the same holds on the level of complexes as  $\Delta'$, $\Delta_0$ and the $\Delta_j^\sigma$ are all subcomplexes of $\Delta$. This shows that the above is a convex ear decomposition of $\Delta'$.
\end{proof}

\section{$f$-, $h$- and $h^*$-vectors of polynomials}
\label{sec:fhh-vectors}

Let $p(k)$ be a polynomial in $k$ of degree $d$. We define the \defn{$h^*$-vector} $h^*(p)=(h^*_0,\ldots, h^*_d)$ of $p$ by
\[
    p(k) =\sum_{i=0}^d h^*_i {k+d-i \choose d}.
\]
Here we make use of the fact that the polynomials ${k+d-i \choose d}$  for $0\leq i\leq d$ form a basis of the vector space of polynomials of degree at most $d$.  Equivalently, the $h^*$-vector can be defined by the equation
\[
(1-z)^{d+1}(p(0)+\sum_{k\geq1} p(k)z^k) = \sum_{i=0}^{d} h^*_i z^i.
\]

Similarly, we define the \defn{$h$-vector} $h(p)=(h_0,\ldots, h_{d+1})$ of $p$ by
\[
    p(k) ={k+d \choose d} + \sum_{i=1}^{d+1} h_i {k+d-i \choose d}.
\]
and $h_0=1$. Here we make use of the fact that  the polynomials ${k+d-i \choose d}$  for $1\leq i\leq d+1$ form a basis of the vector space of polynomials of degree at most $d$. Note that the $h$-vector has length $d+1$ while the $h^*$-vector has only length $d$. This is due to the fact that the first entry of the $h^*$-vector is always $1$. Equivalently, the $h$-vector can be defined by the equation
\[
(1-z)^{d+1}(1+\sum_{k\geq1} p(k)z^k) = \sum_{i=0}^{d+1} h_i z^i.
\]

Finally, we define the \defn{$f$-vector} $f(p)=(f_0,\ldots, f_{d})$ of $p$ by
\[
    p(k) =\sum_{i=0}^{d} f_i {k-1 \choose i}.
\]
Here we make use of the fact that  the polynomials ${k-1\choose i}$, $0\leq i\leq d$ form a basis of the vector space of polynomials of degree at most $d$.

Then the $h$-vector and the $h^*$-vector of a given polynomial are related by
\[
  h^*_i = h_i + (-1)^{d+i}{d+1\choose i} h_{d+1}
\]
for $0\leq i\leq d$ subject to the constraint that $h_0=1$. Similarly, the $f$- and the $h$-vector are related by
\begin{eqnarray}
\label{eqn:f-vs-h}
  h_i = (-1)^{i} {d+1 \choose i} + \sum_{k=0}^{i-1} (-1)^{i-k-1} {d-k \choose i-k-1} f_{k}.
\end{eqnarray}
for $0\leq i\leq d+1$.

Of course we are mainly interested in the case where $p$ is the Ehrhart polynomial of some polytopal complex. Let $\C$ be a $d$-dimensional polytopal complex such that all vertices of $\C$ are lattice points. Let $L_\C$ denote the Ehrhart polynomial of $\C$. Then we define $h^*(\C):=h^*(L_\C)$. When applying the above definition of an $h^*$-vector in this case, it is important to note that $L_\C(0)$ denotes the value of the Ehrhart polynomial at zero and not the value of the lattice point enumerator at zero. If $\C$ is a polytope, then $h^*(\C)$ is the classical $h^*$-vector or Ehrhart-$\delta$-vector of a lattice polytope.

For a $d$-dimensional simplicial complex $\Delta$, the $f$-vector $f(\Delta)=(f_0,\ldots,f_d)$ is classically defined by letting $f_i$ equal the number of $i$-dimensional simplices in $\Delta$. The $h$-vector $h(\Delta)$ is then defined by the relation (\ref{eqn:f-vs-h}). See \cite{Ziegler95, BilleraBjorner2004, Stanley1996}. Note that in this case the Euler characteristic $\chi(\Delta)$ of $\Delta$ is
\[
   \chi(\Delta)=1+(-1)^{d}h_{d+1}(\Delta).
\]
Moreover, if $\Delta$ is unimodular, then $L_\Delta(0)=h^*_0(\Delta)=\chi(\Delta)$, see Lemma~\ref{lem:vectors} below.

To our knowledge $f$- and $h$-vectors of polynomials have previously not been explicitly defined. Our choice of terminology is justified by the following well-known fact.

\begin{lem}
\label{lem:vectors}
Let $\Delta$ be a unimodular triangulation of an integral polytopal complex $\C$.
\begin{enumerate}
\item $f(\Delta)=f(L_\C)$.
\item  $h(\Delta)=h(L_\C)$.
\item If $\C$ is a topological ball, then $h^*(\C)=h(\C)$.
\end{enumerate}
\end{lem}

\begin{proof}
(1) follows directly from the fact that the Ehrhart polynomial of the relative interior of a unimodular $i$-simplex is ${k-1 \choose i}$. (2) follows from (1) by the relation between the $f$- and $h$-vectors. For shellable complexes, see \cite{Ziegler95}, this can also be seen directly by noting that the Ehrhart polynomial of a unimodular $d$-simplex that has $i$ faces removed is ${k+d-i \choose d}$.  (3) follows from $1=\chi(\Delta)=h^*_0(\Delta)$.
\end{proof}

Notice that for a lattice polytope $P$, the constant term of $L_P$ is 1 and hence $h(L_P)=h^*(L_P)$.

\begin{lem}
\label{lem:linearity}
Let $p$ and $q$ be polynomials with  $\deg(p)=\deg(q)=d$.
\begin{enumerate}
\item If $\deg(p+q)=d$, then $h^*(p+q)=h^*(p)+h^*(q)$.
\item If $\deg(p+q)=d-1$, then $h^*_i(p+q)=\sum_{j=0}^i h^*_j(p) + h^*_j(q)$ for $0\leq i \leq d-1$.
\end{enumerate}
\end{lem}

\begin{proof}
(1) is immediate from the definition. (2) can be seen by observing that 
\[
\frac{\sum_{j=0}^{d-1} h^*_j(p+q) z^j}{(1-z)^{d-1}} = \frac{\sum_{j=0}^{d} h^*_j(p) z^j}{(1-z)^{d}} + \frac{\sum_{j=0}^{d} h^*_j(q) z^j}{(1-z)^{d}},
\]
which implies
\[
\sum_{j=0}^{d} (h^*_j(p)+h^*_j(q))z^i = h^*_0(p+q) z^0 + \sum_{j=1}^{d-1} (h^*_j(p+q) - h^*_{j-1}(p+q))z^j - h^*_{d-1}(p+q)z^{d}
\]
from which the claim follows by induction.
\end{proof}

It turns out that the $h^*$-vectors of the polynomials $k^d$ and $(2k+1)^d$ are given by Eulerian and MacMahon numbers, respectively. Given $n\in\NN$ and $0\leq i \leq n$ we define the \defn{Eulerian number} $A(n,i)$ and the \defn{MacMahon number} $B(n,i)$ by
\begin{eqnarray*}
  A(n,i) & = &  \sum_{j=0}^i (-1)^j{n+1 \choose j}(i-j)^n \\
  B(n,i) & = & \sum_{j=1}^i (-1)^{i-j}{n \choose i-j}(2j-1)^{n-1}.
\end{eqnarray*}
These are sequences A008292 and A060187 in the Online Encyclopedia of Integer Sequences \cite{Sloane} with the exception that we also consider $A(n,0)=B(n,0)=0$. If we let $A(n,n+1)=0$, then we have for $0\leq i\leq n$
\begin{eqnarray}
\label{eqn:eulerian-1}
h_i^*(k^n) & = & A(n,i), \\
\label{eqn:eulerian-2} h_i^*((k+1)^n) & = & A(n,i+1), \\
\label{eqn:macmahon} h_i^*((2k+1)^n) & = & B(n+1,i+1).
\end{eqnarray}
All of these identities are straightforward to compute, see also \cite[Section 2.2]{BeckRobins07}.

\section{Enumerative Consequences}
\label{sec:enumerative-consequences}

For positive integers $h$ and $i$ there exists a unique sequence of integers $a_i>a_{i-1}>\ldots>a_j \geq j \geq 1$ such that
\[
  h= {a_i \choose i} + {a_{i-1} \choose {i-1}} + \cdots + {a_j\choose j}.
\]
We then define
\[
  h^{<i>} := {a_i + 1 \choose i + 1} + {a_{i-1} + 1 \choose {i}} + \cdots + {a_j + 1\choose j + 1}.
\]
Now, a sequence of non-negative integers $(h_0,\ldots,h_d)$ is an \defn{$M$-vector} if $h_0=1$ and $h_{i+1} \leq h_{i}^{<i>}$ for all $1\leq i\leq d-1$. We say that a vector $h=(h_0,\ldots,h_d)$ of $d+1$ integers satisfies the \defn{$g$-constraints} if 
\begin{enumerate}
\item $h_0 \leq h_1 \leq \ldots \leq h_{\floor{d/2}}$,
\item $h_i\leq h_{d-i}$ for $i\leq d/2$,
\item $(h_0,h_1-h_0,h_2-h_1, \ldots,h_{\ceil{d/2}}-h_{\ceil{d/2}-1})$ is an $M$-vector.
\end{enumerate}

\begin{thm}[Chari, Swartz]
\label{thm:ced-bounds} Let $\Delta$ denote an abstract $(d-1)$-dimensional simplical complex with a convex ear decomposition. Then the $h$-vector of $\Delta$ satisfies the $g$-constraints.
\end{thm}

This is \cite[Theorem 14]{HershSwartz08}, where the first two constraints are due to Chari \cite{Chari97} and the last constraint is due to Swartz \cite{Swartz06}.

Combining this result with our Theorem~\ref{thm:iop-ced} yields the following result about Erhart polynomials of inside-out polytopes.

\begin{thm}
\label{thm:iop-bounds}
Let $(P,\H)$ be an integral inside-out polytope in which all faces are compressed. Then the $h$-vector of the polynomial $L_P(k) - L_{(P,\H)}(k)$ satisfies the $g$-constraints.
\end{thm}

\begin{proof}
Let $\Delta'\subset\Delta$ be any pulling triangulation of the relative polytopal complex $\C'\subset\C$ associated with $(P,\H)$. By Lemma~\ref{lem:iop-regular}, $\C$ is a regular subdivision and, as $\Delta$ arises from $\C$ by pulling vertices, $\Delta$ is a regular triangulation of $P$. By Theorem~\ref{thm:iop-ced} we conclude that $\Delta'$ has a convex ear decomposition. So, by Theorem~\ref{thm:ced-bounds}, the $h$-vector of $\Delta'$ satisfies the $g$-constraints. All faces of $(P,\H)$ are compressed, so all faces of $\Delta$ are unimodular. Thus, by Lemma~\ref{lem:vectors}, the $h$-vectors of $\Delta'$ and $L_{\Delta'}$ coincide and we conclude that
\[
  h(L_P(k) - L_{(P,\H)}(k)) = h(L_{\Delta'}(k)) =  h(\Delta')
\]
satisfies the $g$-constraints as desired.
\end{proof}

This implies bounds on the coefficients of modular flow and tension polynomials of graphs, by virtue of the fact that in these cases $P$ is a unit cube and thus $L_P(k)=(k+1)^d$.

\begin{thm}
\label{thm:bounds-modular}
Let $p$ denote the modular flow polynomial or the modular tension polynomial of a graph. Let $d$ denote the degree of $p$. Then the $h$-vector of the polynomial $(k+1)^d-p(k)$ satisfies the $g$-constraints.
\end{thm}

\begin{proof}
As we have seen in Section~\ref{sec:flow-and-tension}, the modular flow or tension polynomial of any graph can be realized as the Ehrhart polynomial of an integral inside-out polytope $(P,\H)$ with compressed faces, where $P=[0,1]^d$ and $d$ denotes the degree of the polynomial. The claim now follows from Theorem~\ref{thm:iop-bounds} and the fact that $L_{[0,1]^d}(k)=(k+1)^d$.
\end{proof}

\section{The integral case}
\label{sec:integral-case}

In the case of the integral flow and tension polynomials, the inside-out polytopes $(P,\H)$ are not of the form $P=[0,1]^d$. Rather, $P$ will depend on the graph, so, given a polynomial $p$, we cannot compute a polynomial $p'$ such that if $p$ is a, say, flow polynomial then $p'$ satisfies the $g$-constraints. The best we can say is the following.

\begin{thm}
\label{thm:bounds-integral}
Let $p$ denote the integral flow polynomial or the integral tension polynomial of a graph $G$. Then $L_{F_G}(k)-p(k)$ or $L_{T_G}(k)-p(k)$, respectively, satisfy the $g$-constraints.
\end{thm}

\begin{proof}
As we have seen in Section~\ref{sec:flow-and-tension}, the integral flow or tension polynomial of any graph can be realized as the Ehrhart polynomial of an integral compressed inside-out polytope $(F_G,\H)$ or $(T_G,\H)$, respectively. The claim then follows from Theorem~\ref{thm:iop-bounds}.
\end{proof}

To check whether a given $p$ satisfies this necessary condition, say in the flow case, we would have to check whether $L_{F_G}(k)-p(k)$ satisfies the $g$-constraints for all graphs $G$ with $\Flow_G=p$. So the question arises which polynomials are of the form $L_{F_G}$ or $L_{T_G}$ for some graph $G$: Which polynomials are Ehrhart polynomials of integral flow or tension polytopes? We address this question in this section by giving constraints on the $h^*$-vectors of these polynomials.

First we exploit some nice geometric properties of integral flow and tension polytopes to show that their $h^*$-vectors are palindromic.

\begin{definition}
	A lattice polytope $P$ is  \textbf{reflexive} if 
		\begin{enumerate}
			\item $\inter \ P\cap\Z^d=\{\mathbf{0}\}$;  
			\item $\inter \ ((k+1)P)\cap\Z^d=kP\cap\Z^d$ for all  $k\in\Z_{>0}$.
		\end{enumerate}
\end{definition}

The following proposition gives a method for obtaining new reflexive polytopes from a given reflexive polytope. 

\begin{prop}
 \label{reflex}
 Let $P$ be a reflexive $d$-polytope and let $S$ be a linear subspace of $\R^d$. 
 If $Q:=P \cap S$ is a lattice polytope, then $Q$ is reflexive.
\end{prop}
 
 \begin{proof}
Since $S$ a subspace, $Q^\circ\cap\Z^d=\{\mathbf{0}\}$. 
 Suppose $\x \in \Z^d \cap (k+1) Q^\circ \setminus kQ$ 
 for some $k \in \Z_{>0}$, then $\x$ is also in $(k+1)P^{\circ} \setminus tP$.
This contradicts the fact that $P$ is reflexive.
 \end{proof}

Taking $P = [-1,1]^{m}$ (where $m = |E|$) and letting $S$ be the flow space (tension space, respectively) in Proposition~\ref{reflex} yields

\begin{cor}
\label{Gorenstein}
Integral flow and tension polytopes of a finite graph $G$ are reflexive (and hence Gorenstein).
\end{cor}

The following theorem, due to Hibi, gives a beautiful connection between the geometry of reflexive polytopes on the one hand and palindromic $h^*$-vectors on the other.

\begin{thm}[Hibi \cite{Hibi92}]
\label{palindrom}
If $P$ is a lattice $d$-polytope with the origin in its interior, then 
 $h^{*}(P)$ is palindromic, i.e., it satisfies $h^{*}_{i} = h^{*}_{d-i}$ for $0 \le i \le \lfloor \frac{d}{2} \rfloor$,
 if and only if
$P$ is reflexive.
\end{thm}

Combining Corollary \ref{Gorenstein} and Theorem \ref{palindrom} yields
\begin{thm}
Let $G$ be any finite graph. Then $h^*({F_{G}})$ and $h^*({T_{G}})$ are palindromic. 
\end{thm}

Our next goal is to produce vectors $f_{l},f_{u},t_{l},$ and $t_{u}$ such that the $h^*$-vector of any flow polytope (respectively, tension polytope)  satisfies $f_{l} \le h^* \le f_{u}$ (resp. $t_{l} \le h^* \le t_{u}$). To this end we use Stanley's Monotonicity Theorem (for polytopes):

\begin{thm}[Stanley \cite{Stanley1993}]
\label{thm:monotonicity}
Let $\P \subseteq \mathcal{Q}$ be lattice $d$-polytopes. Then $h^*(P) \le h^*(Q)$.
\end{thm}

To apply this result we use the variants of flow and tension polytopes from Section~\ref{sec:flow-and-tension} that were defined with respect to a fixed spanning forest.

\begin{lem}
\label{alcoved}
Let $G=(V,E)$ be a finite graph. Then any tension polytope $T_G$ of $G$ is a subpolytope of $[-1,1]^{n-1}$. 
Moreover, if $H \subseteq G$ is a subgraph of $G$, then $T_{G}$ is a subpolytope of $T_{H}$ for a suitable choice of spanning forest.
\end{lem}

\begin{proof}
By the construction given in Section~\ref{sec:flow-and-tension}, $T_G\subset [-1,1]^n$.

Now, let $H$ be a subgraph of $G$ such that $V_{H} = V_{G}$. Let $T$ be a spanning forest of $H$ and let the tension polytopes $T_G$ and $T_H$ both be constructed with respect to $T$. As $E_G\setminus T\supset E_H\setminus T$, the set of inequalities defining $T_G$ is a superset of the set of inequalities defining $T_H$. Thus $T_G \subset T_H$.
\end{proof}

We can now give upper and lower bounds on $L_{T_G}$.

\begin{thm}
\label{thm:integral-tension-bounds}
Let $G$ be a connected finite graph with $n$ vertices. Then the $h^*$-vector of $T_G$ satisfies
$$
A(n,i+1) = h_i^*((k+1)^n-k^n) \le h_i^{\star}(T_{G}) \le h_i^*((2k+1)^{n-1}) = B(n,i+1)
$$
for $0\leq i \leq n-1$ and these bounds are tight for all $n$.
\end{thm}

Note that if $G$ has $c$ components $G_1,\ldots,G_c$, then $L_{T_G}=\prod_{i=1}^c L_{T_{G_i}}$.

\begin{proof}
By Lemma~\ref{alcoved} we know that $T_G\subset [-1,1]^{n-1}$ and by Theorem~\ref{thm:monotonicity} and equation (\ref{eqn:macmahon}) we conclude that $h^*_i(T_G) \leq h_i^{\star}(L_{[-1,1]^{n-1}}) = B(n,i+1)$ for all $0\leq i \leq n-1$.
This bound is realized as $[-1,1]^{n-1}$ is the tension polytope of any tree on $n$ vertices.

By Theorem \ref{alcoved}, we have that, for any spanning forest $T$ of $G$, $T_{K_{n}} \subseteq T_{G}$.
Thus to obtain the lower bound we must show that $h^*(L_{T_{K_{n}}}) = A(n,i+1)$.

First we show that $L_{T_{K_{n}}} = (k+1)^{n} - k^{n}$.
To see this note that $L_{T_{K_{n}}}$ counts the number of $(k+1)$-tensions on $K_{n}$.
The $(k+1)$-tensions of $K_{n}$ are in bijection with the functions $c$ in the set
\[
C = \{ c:V\rar \Z \mid \left|c(v_{i}) - c({v_{j}})\right| \le k, \min_{v \in V} c(v) = 0\}.
\]
Because we are dealing with the complete graph the functions $c\in C$ take only values in $\{0,\ldots,k\}$ and for functions $c: V \to \{0,1,\dots, k\}$ the condition $\left|c(v_{i}) - c(v_{j})\right| \le k$ is automatically satisfied. So we can write $C$ as
\begin{align*}
C &= \{c: V \to \{0,1,\dots, k\} \mid \min_{v \in V} c(v) = 0\} \\
 &= \{c: V \to \{0,1,\dots, k\}\} \setminus \{c: V \to \{1,\dots, k\}\}
\end{align*}
and thus
\[
L_{T_{K_n}}(k) = |C| = (k+1)^n-k^n.
\]
Combining Lemma~\ref{lem:linearity} with (\ref{eqn:eulerian-1}) and (\ref{eqn:eulerian-2}) yields $h_i^*((k+1)^n-k^n)=A(n,i+1)$ for all $0\leq i\leq n-1$.
\end{proof}

We now turn our attention to integral flow polytopes. 
If $G$ is a planar graph and $G^{\star}$ is its dual, then, by the (vector space) duality of the flow and tension spaces, the flow polytope of $G$ is the tension polytope of $G^{\star}$. 
In this case, we can apply the theorem above to obtain bounds on the flow polynomial of $G$.

In the non-planar case we proceed as follows.  Let $\diamondsuit_{d}$ denote the $d$-dimensional cross-polytope defined by $$\diamondsuit_{d} = \left\{ \mathbf{x}_{i} \in \R^{d} \mid \sum_{i=1}^{d} |x_{i}| \le 1\right\}.$$

\begin{lem}
\label{alcoved-flow}
Let $G=(V,E)$ be a finite graph. Then for any flow polytope $F_G$ of $G$ we have that $\diamondsuit_{d}$ is a subpolytope of $F_G$  and $F_G$ is a subpolytope of $[-1,1]^{n-1}$ for any choice of spanning forest.
\end{lem}

\begin{proof}
By the construction given in Section~\ref{sec:flow-and-tension}, $F_G\subset [-1,1]^n$. Moreover, for any standard unit vector $e_i$ we have that $C e_i$ is a $\{0,\pm1\}$-vector. Thus $e_i$ and $-e_i$ are contained in $F_G$.
\end{proof}

As in the case of tensions, this implies upper and lower bounds on $h^*(L_{F_G})$.

\begin{thm}
\label{thm:integral-flow-bounds}
Let $G=(V,E)$ be a connected finite graph and let $r=|E|-|V|+1$. Then the $h^*$-vector of $F_G$ satisfies
$$
{r \choose i} \le h_i^{\star}(F_{G}) \le h_i^*((2k+1)^{r}) = B(r+1,i+1)
$$
for $0\leq i \leq n-1$. The upper bound is tight for all $r$.
\end{thm}

\begin{proof}

By Lemma~\ref{alcoved-flow}, the $r$-dimensional cross-polytope, $\diamondsuit_{r}$, is a subpolytope of $\bar F_{G}$. 
So by Stanley's monotonicity theorem we have $h^*(L_{\diamondsuit_{r}}) \le h^*(L_{\bar F_{G}})$.
The lower bound in the theorem stems from the well-known fact that $h_i^{\star}(L_{\diamondsuit_{r}}) = {r \choose i}$, see \cite[Theorem 2.7]{BeckRobins07}.
Since $h_i^{\star}(L_{[-1,1]^{r}}) = h_i^*((2k+1)^r) = B(r+1,i+1)$ by (\ref{eqn:macmahon}), the upper bound follows from Lemma~\ref{alcoved-flow} as well. 
\end{proof}

The upper bound in the above theorem is tight since, for any $r\in\Z_{\>0}$,  
the  flow polytope of the graph consisting of  a single vertex and $r$ loops is $[-1,1]^{r}$. 
The lower bound, on the other hand, is not tight.
For example, no $3$-dimensional flow polytope is lattice isomorphic to $\diamondsuit_{3}$.

\bibliographystyle{alpha}
\bibliography{article}
\label{sec:biblio}

\end{document}